\documentclass[onefignum,onetabnum]{siamart171218}
\pdfoutput=1

\usepackage{listings}
\usepackage{lipsum}
\usepackage{amsfonts}
\usepackage{graphicx}
\usepackage{epstopdf}
\usepackage{algorithmic}
\usepackage{color}
\usepackage{subfigure}
\ifpdf
  \DeclareGraphicsExtensions{.eps,.pdf,.png,.jpg}
\else
  \DeclareGraphicsExtensions{.eps}
\fi

\newsiamremark{remark}{Remark}
\newsiamremark{hypothesis}{Hypothesis}
\crefname{hypothesis}{Hypothesis}{Hypotheses}
\newsiamthm{claim}{Claim}

\headers{A Rank Revealing Factorization Using Arbitrary Norms}{R. Atcheson}

\title{A Rank Revealing Factorization Using Arbitrary Norms}

\author{Reid Atcheson\thanks{Numerical Algorithms Group
  (\email{reid.atcheson@nag.com}).}}

\usepackage{amsopn}

\DeclareMathOperator*{\argmax}{arg\,max}
\DeclareMathOperator*{\argmin}{arg\,min}
\DeclareMathOperator*{\linearspace}{span}
\newtheorem{conjecture}{Conjecture}

\ifpdf
\hypersetup{
  pdftitle={A Rank Revealing Factorization based on the l1 norm}
  pdfauthor={Reid Atcheson}
}
\fi

\begin{document}

\maketitle

\begin{abstract}

  The classic rank-revealing QR factorization factorizes a matrix $A$ as $AP=QR$ where $P$ 
  permutes the
  columns of $A$, $Q$ is an orthogonal matrix, and $R$ is upper triangular with non-increasing
  diagonal entries. This is called rank-revealing because careful choice of $P$ 
  allows the user to truncate the factorization for a low-rank approximation of $A$ 
  with an error term computed in the $l^2$ norm. 
  In this paper I generalize the QR factorization to use any arbitrary norm
  and prove analogous properties for $Q$ and $R$ in this setting. I then show an application
  of this algorithm to compute low-rank approximations to $A$ with error term in the $l^1$ norm
  instead of the $l^2$ norm. I provide Python code for the $l^1$ case as demonstration of
  the idea.

\end{abstract}



\section{Introduction}

Low-rank approximation allows the user to compress an input matrix in a very informative way. 
The low-rank factors can provide useful information about the data which comprises the input
matrix, which forms the basis of Principal Component Analysis (PCA). The gold-standard of
low-rank approximations is the SVD factorization, which gives optimal low-rank approximations 
with respect to the Euclidean norm $\| \cdot \|_2$. 
The problem with SVD is that algorithms for it typically must be iterative in nature, or
even probabalistic. A non-iterative and deterministic algorithm which reveals rank
information can therefore be useful.

The rank-revealing QR factorization \cite{chan_rank_1987} is a deterministic and non-iterative
algorithm which provides rank information on the input matrix by way of the diagonal entries
of its upper triangular factor. It turns out this factorization can in fact be used 
directly for low-rank approximation also, bypassing the SVD entirely, and this has been
exploited heavily in areas such as hierarchical compression of matrices \cite{hackbusch_sparse_1999},\cite{hackbusch_introduction_2002}. Like with the SVD the quality of this low-rank
approximation is often best in the Euclidean norm $\| \cdot \|_2$ because the $QR$ factorization
is explicitly based on the Euclidean dot product. This optimality in the Euclidean norm
has some undesirable properties in other fields however.

For some applications of data analysis the optimality of a low-rank approximation 
in the Euclidean norm results in unfavorable low-rank factors, 
because outliers in data can quickly 
overwhelm the Euclidean norm of that data, resulting in poor approximations. This has
led to the field of "L1 PCA" which tries to find optimal low rank approximations
in the $l^1$ norm instead of the $l^2$ norm 
\cite{tsagkarakis_l1-norm_2017},\cite{markopoulos_efficient_2017},\cite{markopoulos_outlier-resistant_2018}. Unfortunately, since the $QR$ factorization is highly specialized to the Euclidean
norm this suggests that rank-revealing $QR$ strategies can not help in domain. Thus this
new area of low-rank approximation has moved in the direction of iterative or probabalistic
SVD-like algorithms \cite{markopoulos_outlier-resistant_2018}.

In this paper I show that the $QR$ factorization can be generalized to norms other
than the Euclidean norm. I derive the algorithm, state and prove analogous properties
of the resulting $Q$ and $R$ factors, and then show numerical results. This
yields a deterministic and non-iterative algorithm with rank-revealing properties
with the potential to give optimality in norms besides the Euclidean norm.

The paper is organized as follows. The main theory and algorithm is presented in
\cref{sec:main}, an implementation of this algorithm in python
for the special case of the $l^1$ norm is in section \cref{sec:alg}, experimental
results are in \cref{sec:experiments}, and the conclusions follow in
\cref{sec:conclusions}.

\section{Main results}
\label{sec:main}

I start by presenting the algorithm that this paper is based on. This algorithm
accepts a matrix $A\in\mathbb{R}^{m\times m}$ and any norm $\| \cdot \|$ on
$\mathbb{R}^m$ and returns a permutation $P$, 
an upper triangular matrix $R$ with nonincreasing diagonal, and $Q$ such that $AP=QR$. 
I then prove key facts about this algorithm (theorem \ref{thm:bigthm}) and state a 
conjecture (conjecture \ref{con:goodcondition}). I also prove that when 
the input norm is equal to the Euclidean norm, then the factorization reduces to
a classical $QR$ - in the sense that $Q$ becomes orthogonal. This is theorem
\ref{thm:equivalence}. I start first with the algorithm \ref{alg:rrqr} below.

\begin{algorithm}[H] \label{alg:rrqr}
\caption{Arbitrary-norm rank-revealing QR Factorization}

  Start with an input $A\in\mathbb{R}^{m\times m}$ and any norm 
  $\| \cdot \|$ on $ \mathbb{R} ^ m $.
  
  I use the notation $e_k$ to mean the $k$th column of the identity matrix, and I use

  \begin{align}
    A &= (A_1,A_2,A_3,\ldots,A_m) \\
    Q &= (Q_1,Q_2,Q_3,\ldots,Q_m) \\
    P &= (P_1,P_2,P_3,\ldots,P_m) 
  \end{align}

  to represent $A,P$ and $Q$ by their respective columns. Furthermore I define 
  $Q^i \ in \mathbb{R} ^ {m \times i} $ as the first $i$ columns of $Q$:

  \begin{displaymath}
    Q^i = (Q_1,Q_2,\ldots,Q_i).
  \end{displaymath}

  I now define the $P,Q$ and $R$ factors inductively as follows:

  \begin{align} \label{eq:recurrencebase}
    k &=\argmax_i \| A _i \| \\
    P_1 &= e_k \\
    Q_1 &= A_k / \| A_k \| \\
    R(1,1) &= \frac{1}{\| A_k \|}
  \end{align}

  and for any $1\leq j \leq m-1$ I define

  \begin{align}\label{eq:recurrence}
    k_j &=\argmax_i  \min_{c_j \in \mathbf{R}^j }\|  A_i - Q^j c_j \| \\
    c_j &=\argmin _ {c_j \in \mathbf{R}^j }\|  A_{k_j} - Q^j c_j \|  \\
    \gamma_j &= \| A_{k_j} - Q^jc_j \| \\
    P_{j+1} &= e_{k_j} \\
    Q^{j+1} &= (Q^j, \gamma_j ^{-1} ( A_{k_j} - Q^j c_j) ) \\
    R(j,1:j-1) &= c_j \\
    R(j,j) &= \gamma_j 
  \end{align}

\end{algorithm}

The key theoretical result of this paper is summarized in theorem \ref{thm:bigthm}. Following
this theorem is a conjecture which seems true based on numerical evidence supporting it
(see section \cref{sec:experiments}) but a full proof remains elusive. Finally
I prove in theorem \ref{thm:equivalence} that if $\| \cdot \|$ = $\| \cdot \|_2$
then algorithm \ref{alg:rrqr} outputs $Q$ as orthogonal.

\begin{theorem}[Arbitrary-norm Rank-Revealing $QR$ factorization ]\label{thm:bigthm}

  Suppose that $A\in\mathbb{R}^{m \times m}$,$\| \cdot \|$ is a norm, and that 
  $P,Q,R$ are output by algorithm \ref{alg:rrqr}.

  Then the following properties hold:

  \begin{equation} \label{eq:factorization}
    AP=QR
  \end{equation}
  \begin{equation} \label{eq:triangular}
    R \text{ is upper triangular with nonincreasing diagonal entries}
  \end{equation}
  There exists a constant $C_1>0$ independent of $A$ such that
  \begin{equation} \label{eq:normbound}
    \max _ { \| x \| = 1 } \| Qx \| \leq C_1
  \end{equation}

\end{theorem}

\begin{conjecture}[Inverse Bound] \label{con:goodcondition}
  Suppose that $A\in\mathbb{R}^{m \times m}$,$\| \cdot \|$ is a norm, and that 
  $P,Q,R$ are output by algorithm \ref{alg:rrqr}.

  Then there exists a constant $C_2>0$ that depends only on the norm $ \| \cdot \| $ such that

  \begin{equation} \label{eq:invnormbound}
    \min _ { \| x \| = 1 } \| Qx \| \geq C_2
  \end{equation}

\end{conjecture}

Properties \ref{eq:factorization} and \ref{eq:triangular} are standard and precisely
match the classical $QR$ factorization with column pivoting.
Properties \ref{eq:normbound} and \ref{eq:invnormbound} perhaps
require more explanation. In the classical $QR$ factorization the matrix $Q$ is orthogonal
($Q^TQ = I$). Strictly speaking we could insist that $Q$ also be orthogonal in the above theorem,
but the utility of orthogonality is lost when using norms different from the $l^2$ norm. This
utility stems from the fact that the $l^2$ norm is derived from an inner product,
so orthogonality has strong implications on the conditioning of $Q$ in this norm.

Thus to  find an analogue to orthogonality I require that the matrix $Q$ be 
\emph{well conditioned}. The bounds \ref{eq:normbound} and \ref{eq:invnormbound} prove
that $Q$ is invertible (full-rank), but also that the conditioning of $Q$ does not
depend on the conditioning of $A,$ which the theorem allows to be highly numerically singular.
By way of example, if we were to state this theorem for $ \| \cdot \|  = \| \cdot \|_2 $ then
we would actually have $C_1 = C_2 = 1$.

I now prove theorem \ref{thm:bigthm}, minus the conjecture:
\begin{proof}

  To prove equation \ref{eq:factorization} note that $AP^1 = Q^1R(1,1)$ follows directly
  from the base case definitions of these quantities. Now assume $AP^j = Q^jR(1:j,1:j)$  
  for some $j.$ Then

  \begin{align*}
    Q^{j+1}R(1:j+1,1:j+1) &= (Q^j,Q_{j+1})
    \begin{bmatrix}
       R(1:j,1:j) & c_j \\
      0              & \gamma_j
    \end{bmatrix} \\
    &= (Q^j R(1:j,1:j),Q_{j}c + \gamma_j Q_{j+1}) \\
    &= (A P^j,A_{k_j}) \\
    &= A P^{j+1}
  \end{align*}

  For \ref{eq:triangular} it's clear that $R$ is upper triangular, but to show that its
  diagonal entries are nonincreasing observe that from the optimality property of $c_j$ we have
  \begin{align*}
    R(1,1) &=    \argmax_i \| A_i \| \\
           &\geq \| A_{k_1} \| \\
           &\geq \| A_{k_1} - Q^1 c \| \\
           &= R(2,2)
  \end{align*}

  and for any $j>1$:

  \begin{align*}
    R(j,j) &=    \max_i \min _ {c_j \in \mathbf{R}^j }\|  A_i - Q^j c_j \|  \\
           &\geq \max_i \min _ {c_{j+1} \in \mathbf{R}^{j+1} }\|  A_i - Q^{j+1} c_{j+1} \|  \\
           &= R(j+1,j+1)
  \end{align*}
  and finally for the conditioning properties \ref{eq:normbound} and \ref{eq:invnormbound}
  observe that if $\|x\|=1$ then
  \begin{align*}
    \|Qx\| &=    \| \sum _{i=1}^m Q_i x_i \| \\
           &\leq  \sum_{i=1}^m \| Q_i x_i \| \\
           &\leq  \sum_{i=1}^m \| Q_i\| \| x_i \| \\   
           &\leq  \| x \| _1 
  \end{align*}
  where the final inequality is a consequence of Holder's inequality. Finally we may apply norm
  equivalence between all norms in finite dimensional spaces to choose $C_1>0$ such that
  $\| x \| _1 \leq C_1 \|x \|$ holds for all $x$ to complete the proof of \ref{eq:normbound}.
  The bound \ref{eq:invnormbound} remains conjecture, but is supported by numerical evidence
  in section \cref{sec:experiments}
\end{proof}

\begin{theorem}[Classic QR as Special Case ]\label{thm:equivalence}
  Suppose that $A\in\mathbb{R}^{m \times m}$,$\| \cdot \| _ 2$ is the
  $l^2$ norm, and that  $P,Q,R$ are output by algorithm \ref{alg:rrqr}.

  Then $Q$ is orthogonal, i.e. $Q^TQ = I$.
\end{theorem}
\begin{proof}
  By the inductive definition of $Q$ in \ref{eq:recurrence} we have
  \begin{equation}
    Q^{j+1} = (Q^j, \gamma_j ^{-1} ( A_{k_j} - Q^j c_j) )
  \end{equation}
  Recall that $c_j$ solves the minimization problem
  \begin{equation}
    c_j =\argmin _ {c_j \in \mathbf{R}^j }\|  A_{k_j} - Q^j c_j \|  \\
  \end{equation}
  which means it is forming the $l^2$ projection of $A_k$ onto the space $V=\linearspace(Q_1,\ldots,Q_j)$.
  Since $Q_{j+1}$ is the residual of this projection, it is orthogonal to the whole space $V$.
\end{proof}

\section{Implementation for $l^1$ norm using linear programming} \label{sec:alg}

The key ingredient of algorithm \ref{alg:rrqr} is the ability
to compute solutions to minimum-norm linear problems such as $\argmin \| b - Ax \|$. 
For the $l^2$ case there are already established and robust algorithms for this problem, but
it's less obvious for other norms. For the $l^1$ norm we can cast it as a linear program. In 
other words:

\begin{displaymath}
  \argmin_x \| b - Ax \|_1
\end{displaymath}

is equivalent to the linear program

\begin{align*}
  \argmin_t \sum _{i=1}^m t _ i \\
  \text{Subject to} \\
  b - Ax &\leq t \\
  Ax  - b &\leq t \\
  t &\geq 0 
\end{align*}

This is implemented using Python in the 
appendix at listing \ref{impl:opensource}, this uses the open source tools 
NumPy \cite{walt_numpy_2011} and SciPy \cite{jones_scipy:_2001}. 

The linear program approach is correct but does not seem 
to scale well for larger matrices $A$ Thus I also provide
a Python implementation that uses the NAG numerical
library \cite{naglibrary} in listing \ref{impl:nag}. 
Furthermore I have also implemented the relations \ref{eq:recurrencebase} and \ref{eq:recurrence}
as a function in python in listing \ref{impl:rrqr}. This implementation can use
either the NAG $l^1$ solver from \ref{impl:nag} or the open source $l^1$ solver
from \ref{impl:opensource} by changing the value of the \emph{l1alg} parameter.

I now proceed to show numerical results of this algorithm.

\section{Experimental results}
\label{sec:experiments}

The results below are designed to validate some of the theoretical properties proven
and asserted earlier. These include properties like the well-conditioning of $Q$ and the
non-increasing property for the diagonal of $R$. I also include results on
low-rank approximation from this factorization as that was the primary motivation
of deriving this algorithm.

\subsection{Diagonal Entries of R}

These experiments test the theorem result \ref{eq:triangular}. Here I take 
$A \in \mathbb{R}^{m \times m }$ constructed explicitly as an SVD factorization
$A = U\Sigma V^T$ with diagonal entries of $\Sigma$ varying in relative size, which I indicate
with $\sigma_m \argmin_i \Sigma _{i,i}$ and $\sigma_1 = \argmax_i \Sigma_{i,i}$.

If $R$ truly has rank-revealing properties then it should exhibit rapid decay of diagonal entries
when $A$ becomes progressively more singular.

\begin{figure}[H]
  \centering
  \begin{minipage}[b]{0.4\textwidth}
        \includegraphics[scale=0.3]{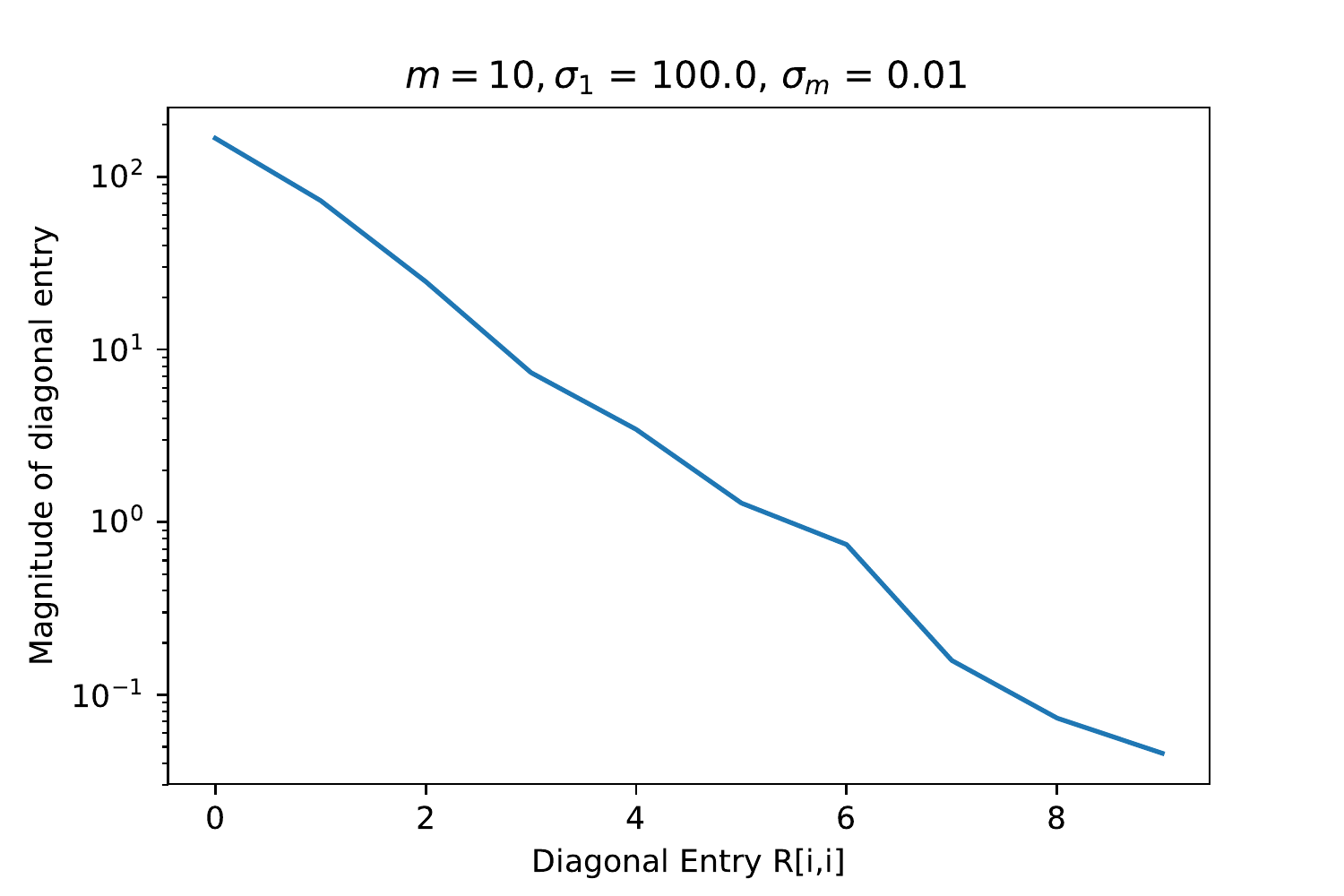}
        \includegraphics[scale=0.3]{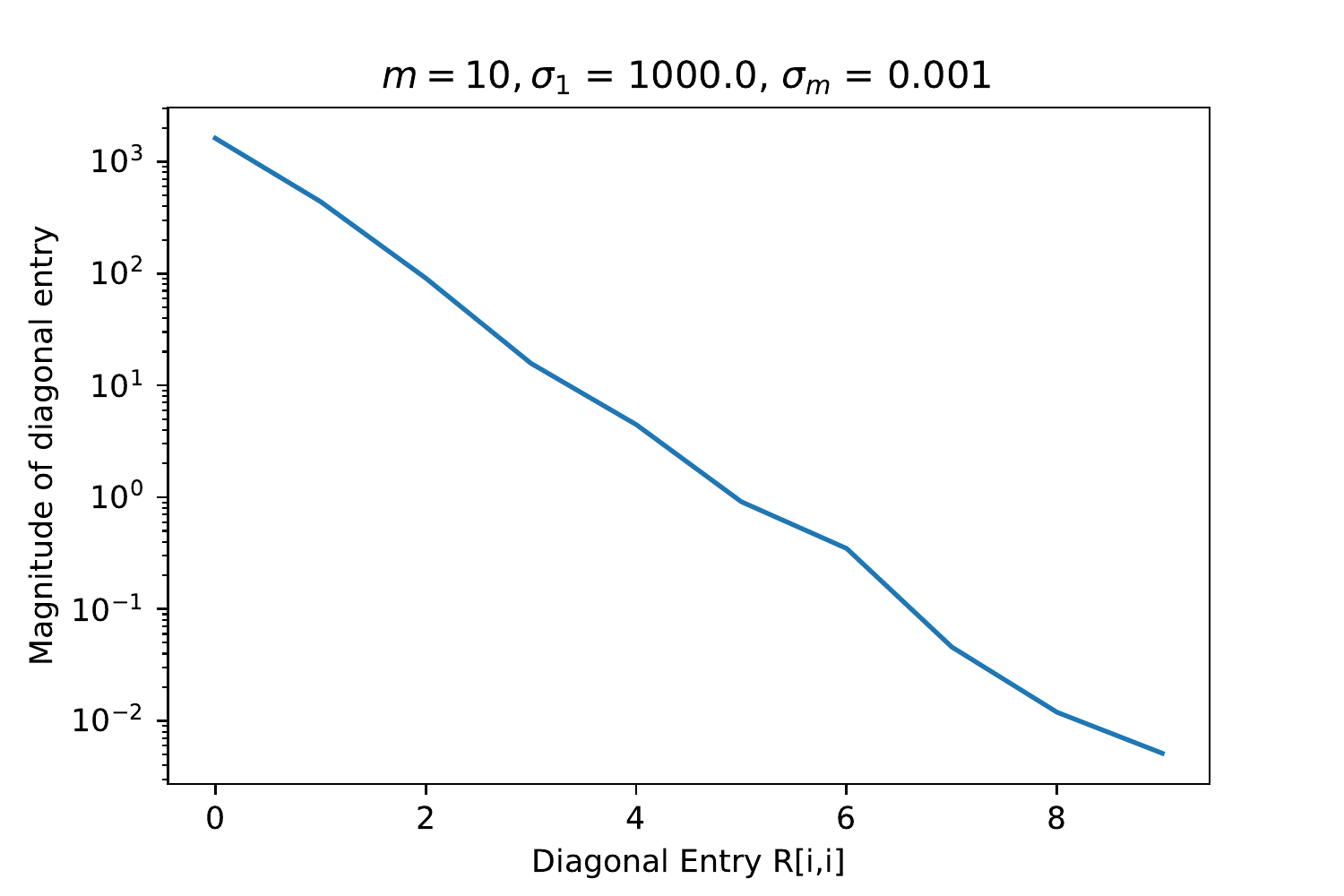}
        \includegraphics[scale=0.3]{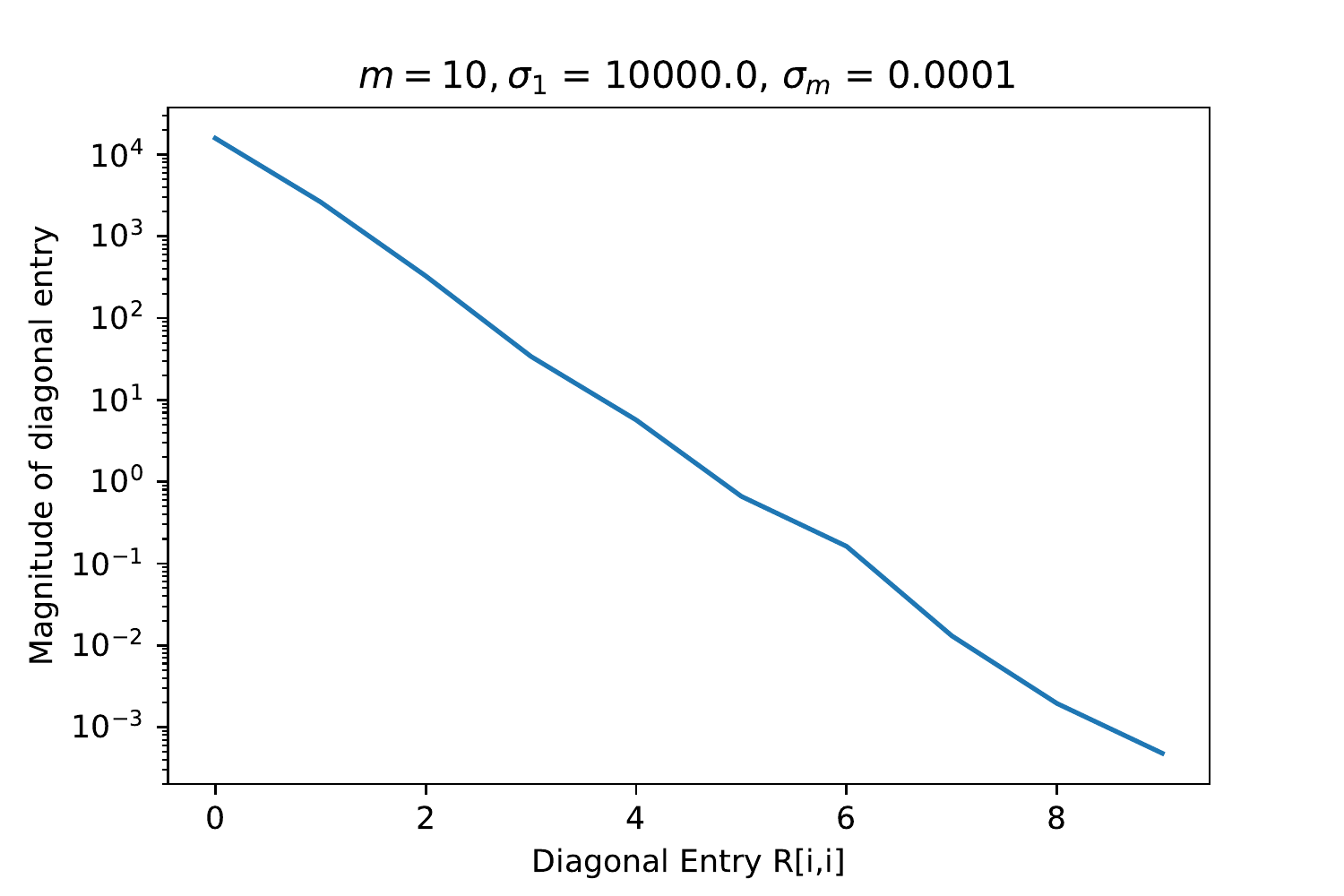}
        \caption{Decay of diagonal entries for $m=10$}
  \end{minipage} \qquad
  \begin{minipage}[b]{0.4\textwidth}
        \includegraphics[scale=0.3]{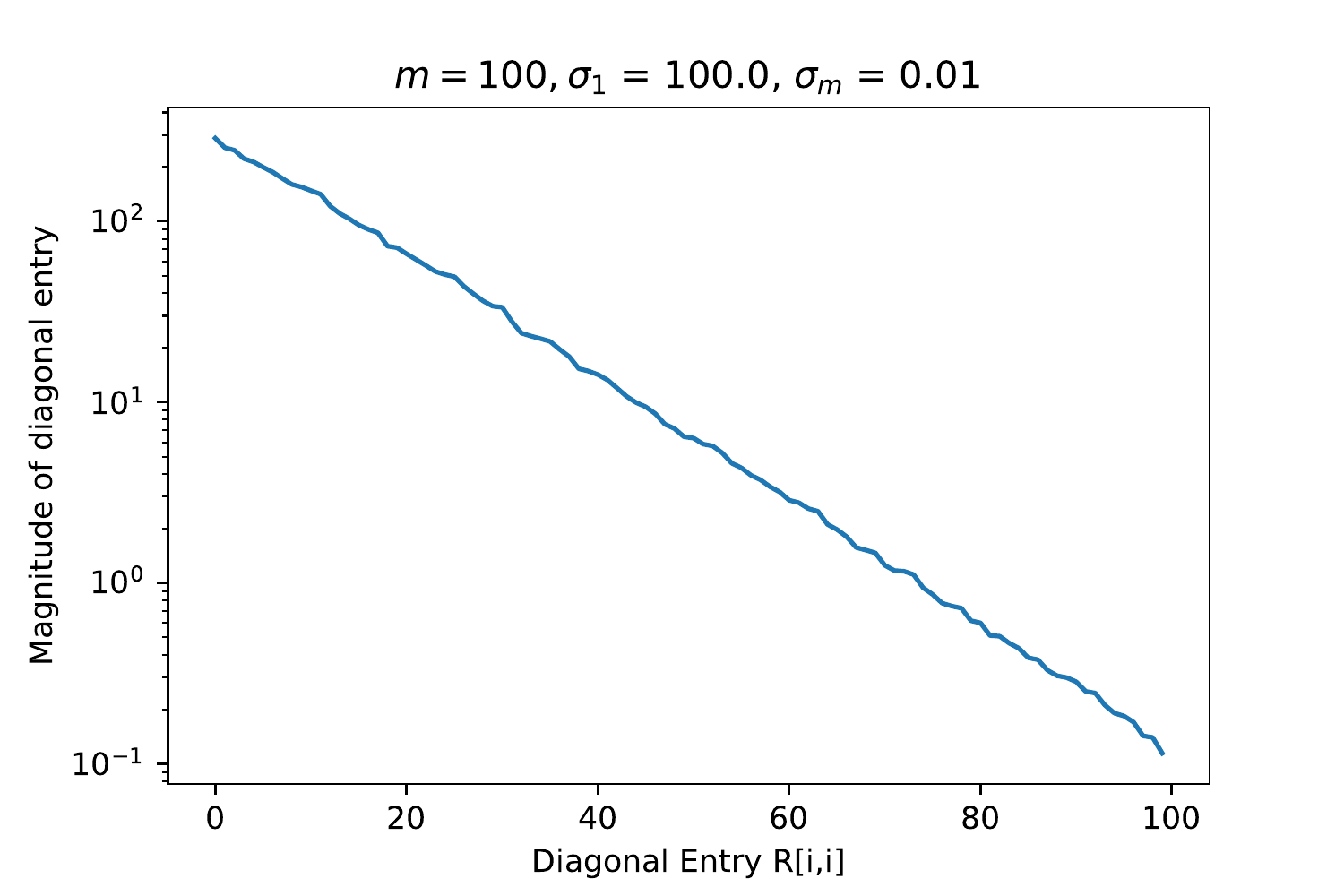}
        \includegraphics[scale=0.3]{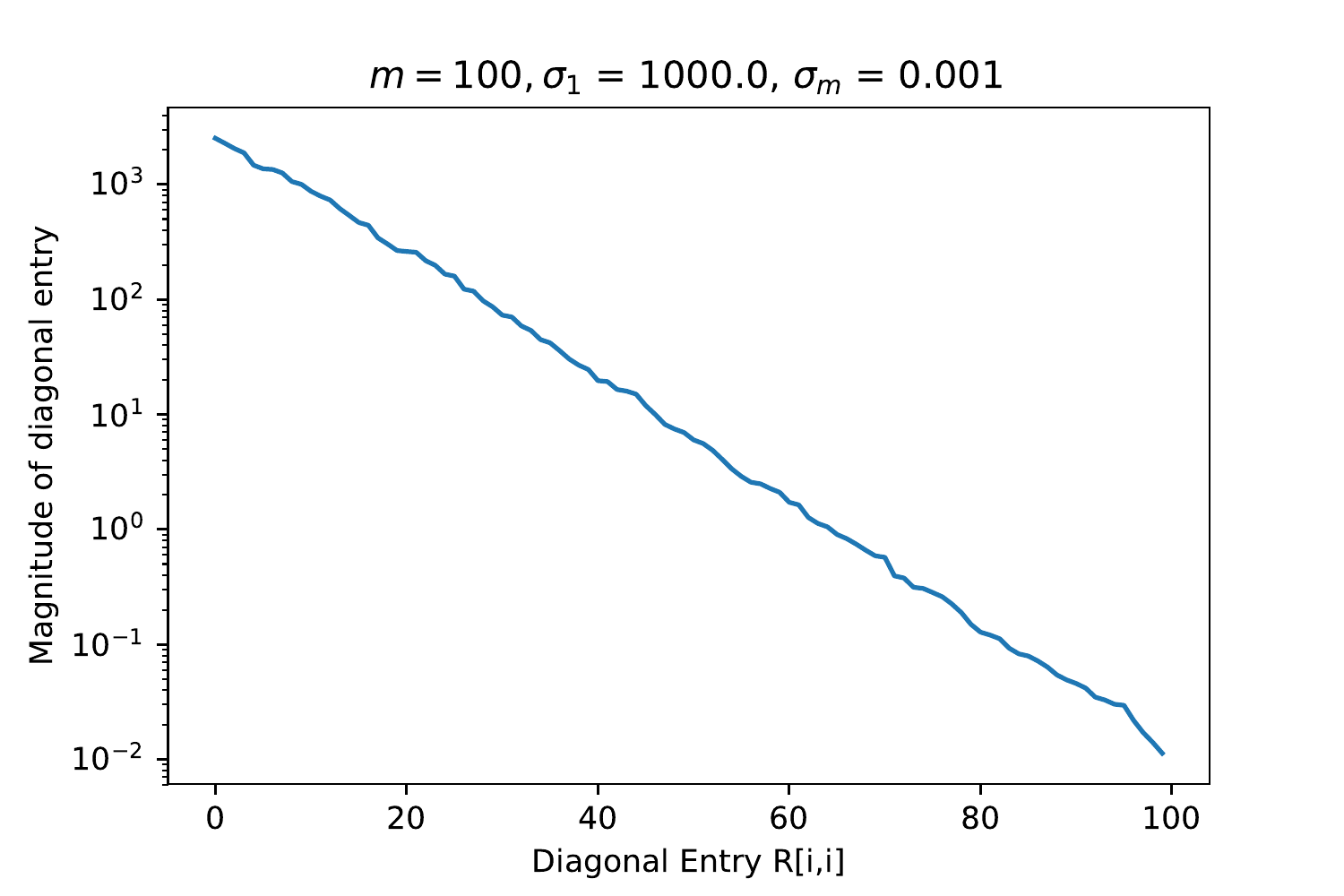}
        \includegraphics[scale=0.3]{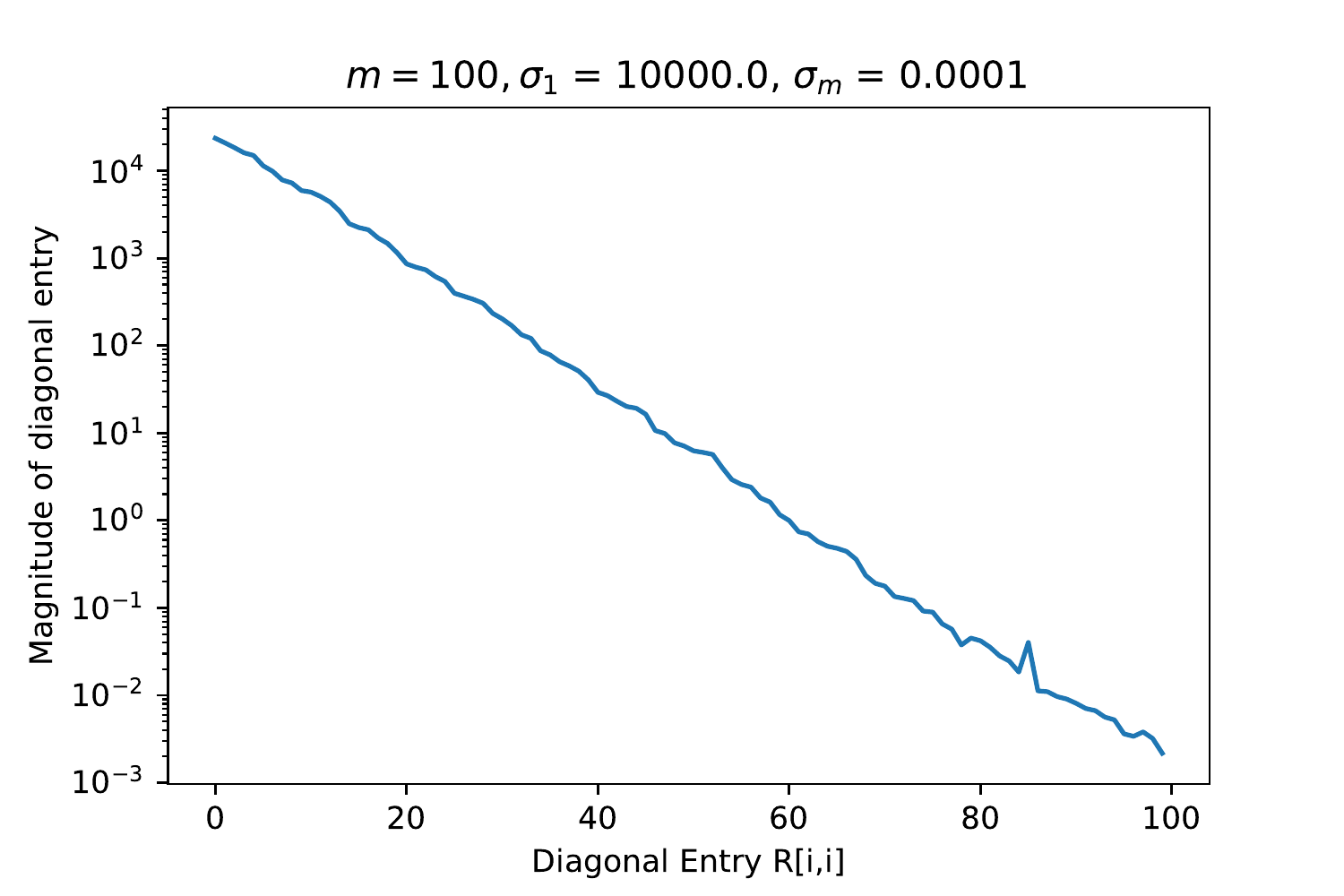}
        \caption{Decay of diagonal entries for $m=100$}
  \end{minipage}
\end{figure}

These results suggest that $R$ is capturing low rank information.

\subsection{Conditioning of Q}
An important part of successful rank-revealing factorization $AP=QR$ 
is the conditioning of $Q$ should be independent of the conditioning of $A$.
The key theoretical result which would prove this would be \ref{eq:invnormbound}, but
unfortunately I was unable to prove this. Here I give numerical evidence that it
does appear to be true.

I take $A \in \mathbb{R}^{m \times m }$ constructed explicitly as an SVD factorization
$A = U\Sigma V^T$ with diagonal entries of $\Sigma$ varying in relative size. I compute the
condition numbers $\|A\|_1 \|A^{-1} \|_1$,$\|Q\|_1 \|Q^{-1} \|_1$ and plot them against each other
in figure \ref{fig:cond}.

\begin{figure}[H]
  \centering
  \includegraphics[scale=0.5]{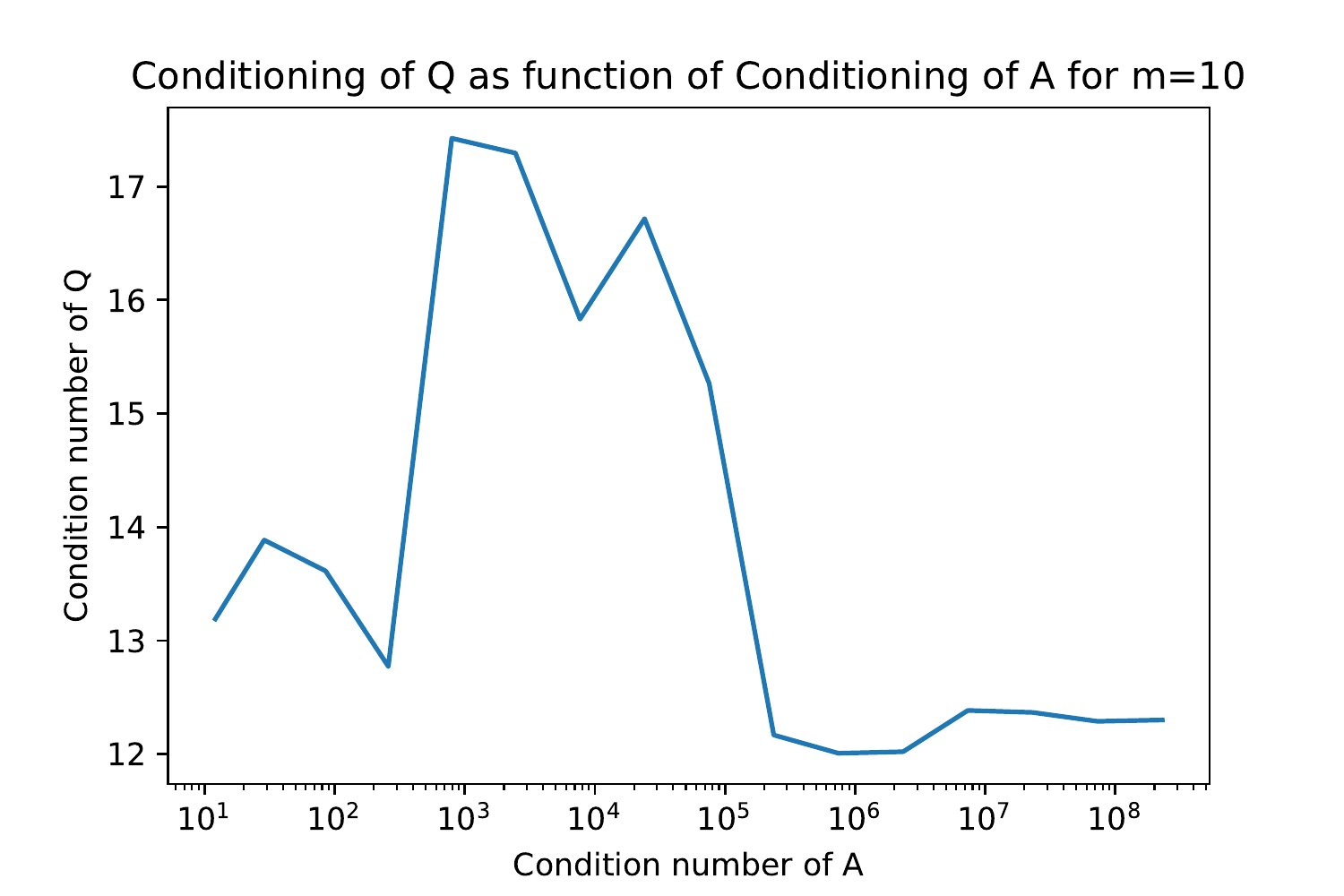}
  \includegraphics[scale=0.5]{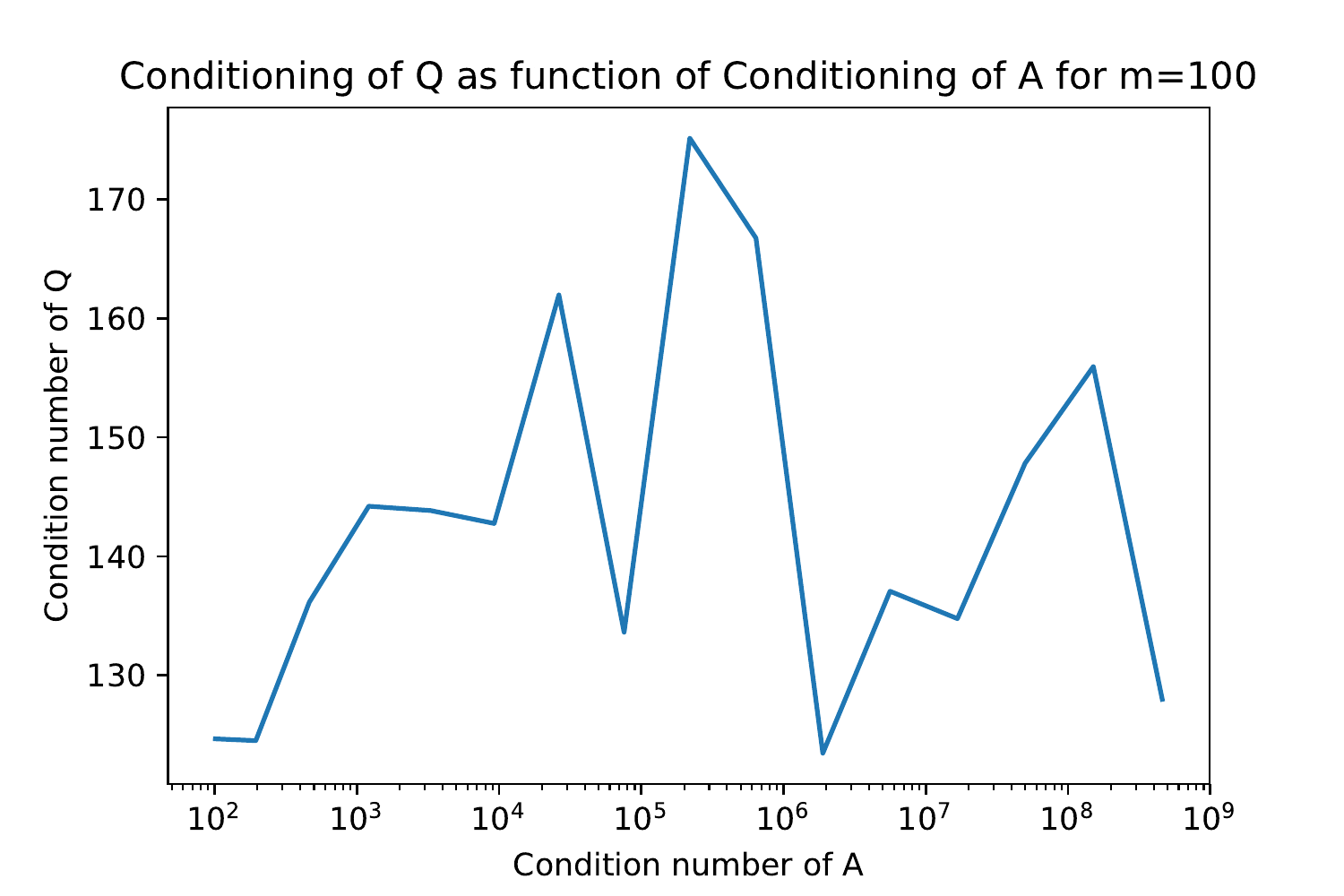}
  \caption{Conditioning of $Q$ compared to conditioning of $A$ for $m=10$ and $m=100$}
  \label{fig:cond}
\end{figure}

\subsection{Factorization error}
Next I illustrate that the factorization error $\|AP - QR\|$ also does not depend on the
conditioning of $A$.  

I take $A \in \mathbb{R}^{m \times m }$ constructed explicitly as an SVD factorization
$A = U\Sigma V^T$ with diagonal entries of $\Sigma$ varying in relative size. I compute the
condition numbers $\|A\|_1 \|A^{-1} \|_1$, and factorization errors
$ \|AP - QR \|_1 $ and plot them against each other in figure \ref{fig:error}

\begin{figure}[H]
  \centering
  \includegraphics[scale=0.5]{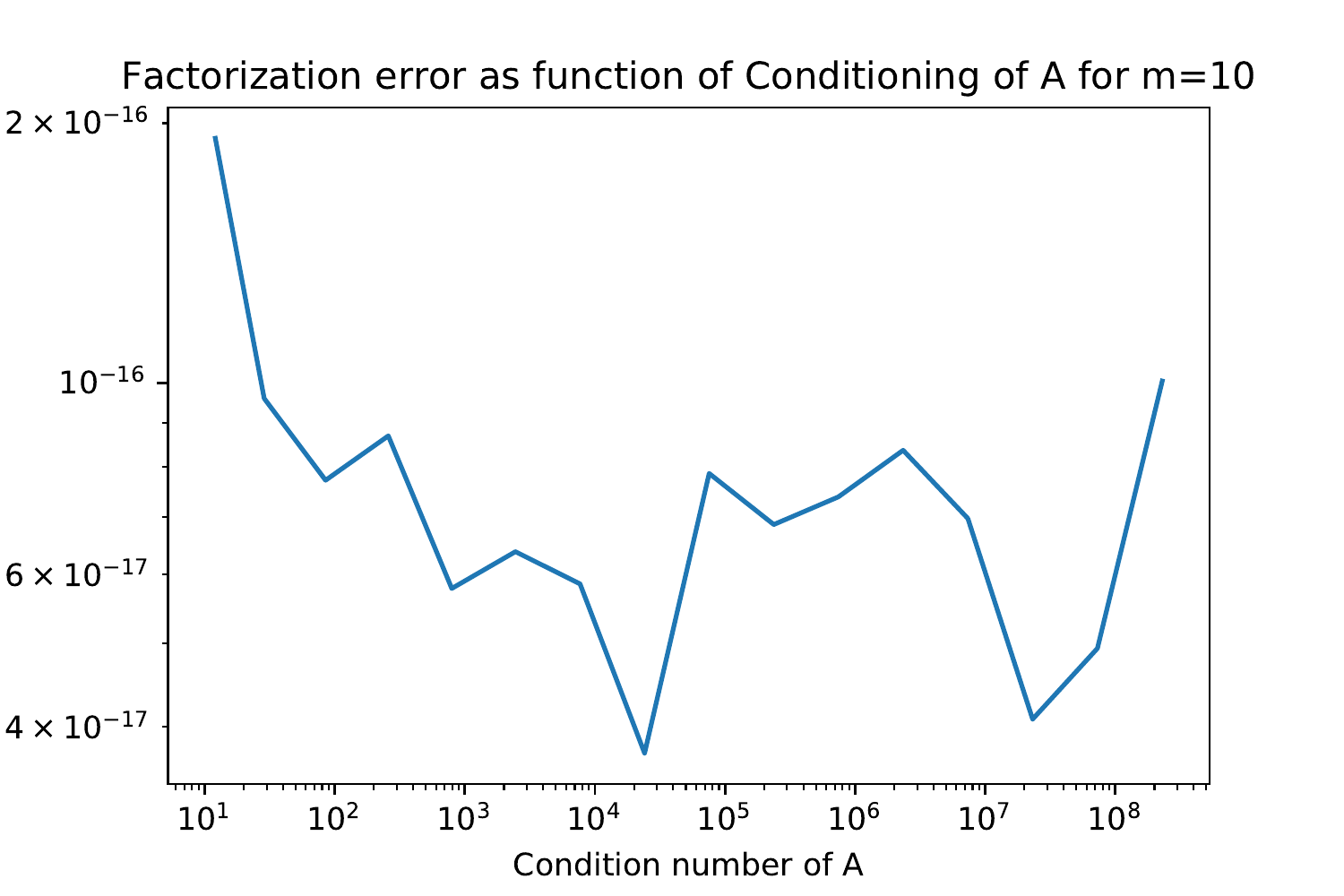}
  \includegraphics[scale=0.5]{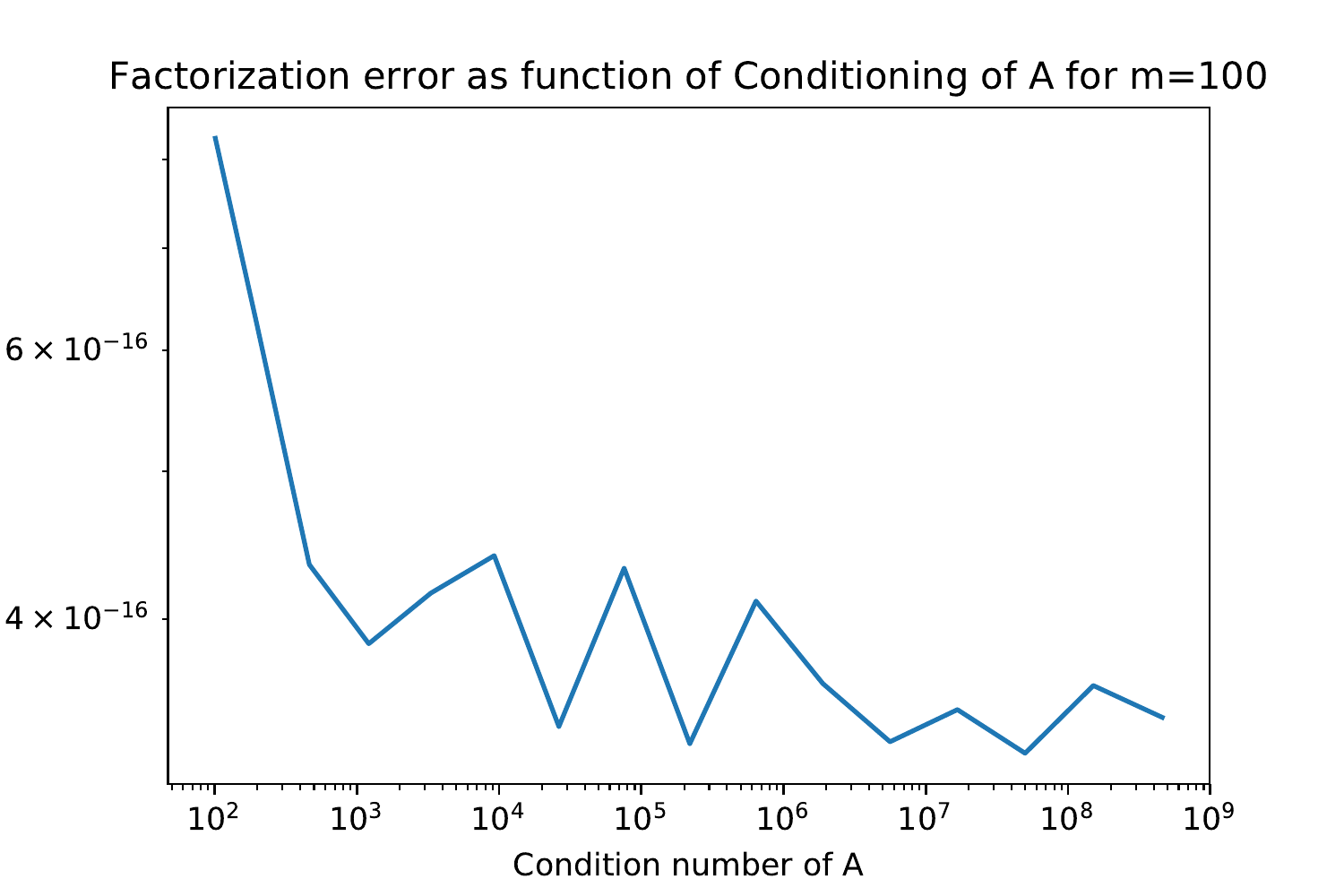}
  \caption{Factorization error $\|AP - QR \|_1$ for $m=10,100$.}
  \label{fig:error}
\end{figure}

\subsection{Low-rank approximation}

With the rank-revealing properties validated I now show an example of low-rank approximation.

For this test I again generate $A$ by forming it as an explicit SVD factorization 
$A = U\Sigma V^T$ with $\max_i \Sigma_{i,i} =1$ and $\min_i \Sigma{i,i}=10^{-6}$. I then
compute two factorizations of $A$:

\begin{align}
  AP_1 &= Q_1R_1 \text{ l1 RRQR factorization } \\
  AP_2 &= Q_2R_2 \text{ Classic RRQR factorization}
\end{align}

Next I truncate the factorizations to be a rank-$k$ approximation to $A$ as follows:

\begin{align}
  A &\approx Q_1(:,1:k)R_1(1:k,:)P_1^T \\
  A &\approx Q_2(:,1:k)R_2(1:k,:)P_2^T
\end{align}

For the first study I compare the induced $l^1$ matrix norm
error of these approximations for $k=1,\ldots,60$
this is in figure \ref{fig:lowrankerror}.

\begin{figure}[H]
  \centering
  \includegraphics[scale=0.5]{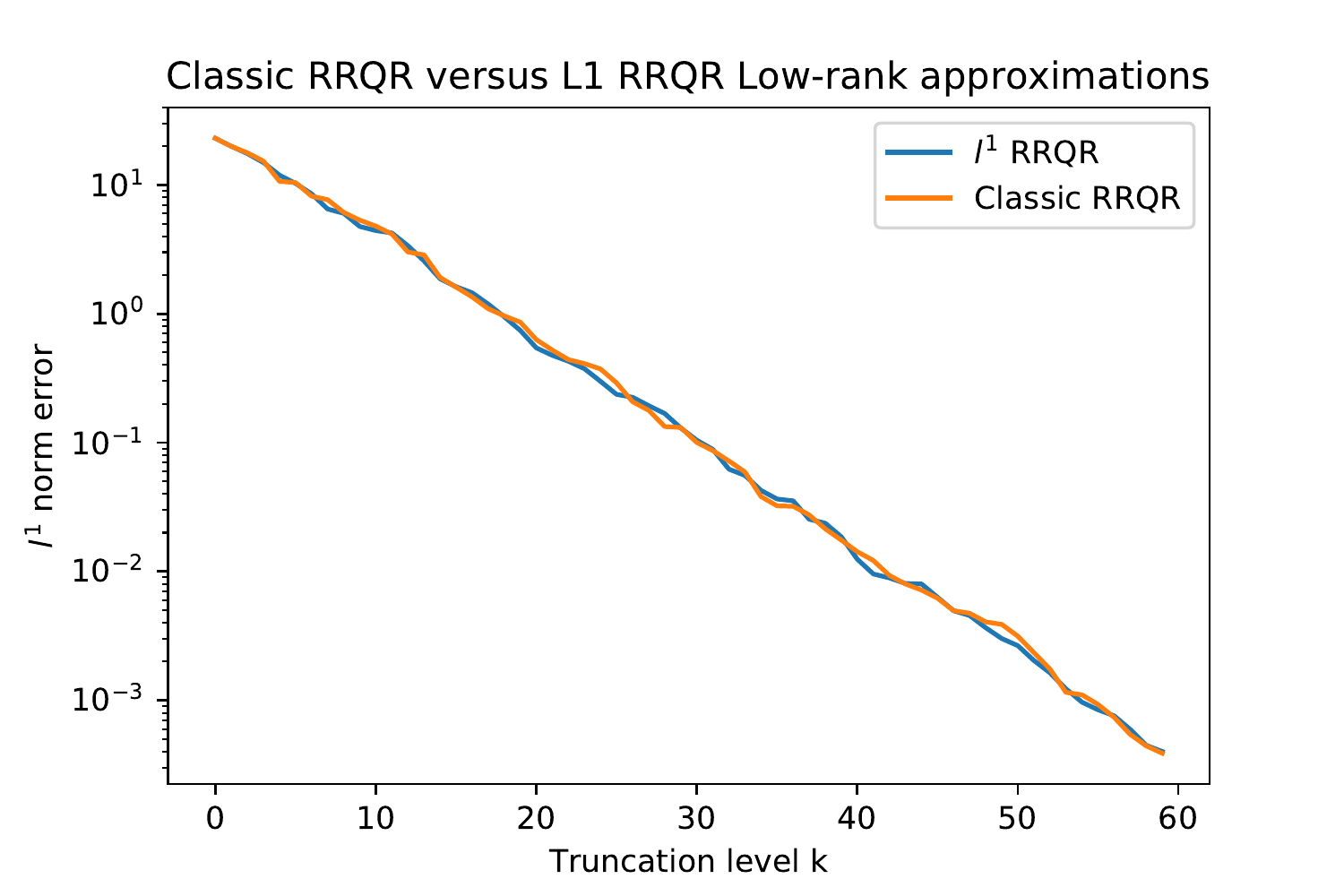}
  \caption{Low-rank approximation error for rank $k=1,\ldots,60$ and $m=100$}
  \label{fig:lowrankerror}
\end{figure}

This result suggests that there is little difference in results between $l^1$ and $l^2$
rank revealing factorizations. The next section \ref{sec:outlier} shows the subtle
difference between $l^1$ and $l^2$ norms for low-rank approximations, and why 
the $l^1$ norm may be preferred in some situations.

\subsection{Resistance of $l^1$ norm to outliers} \label{sec:outlier}

One of the original motivations for deriving algorithm \ref{alg:rrqr}
was to be able to do $l^1$ low-rank approximations, which can be 
very robust with respect to outliers in data \cite{candes_robust_2011}.

I show here that to some extent this appears to be reflected in the $l^1$ version
of the rank revealing factorization. To illustrate this I first show a "clean" example
without outliers, and then do rank $k$ approximations for $k=1,2,3$ for both
classical RRQR and $l^1$ RRQR. Then I introduce outliers to this same data and show
the classical RRQR algorithm quickly is drawn to over-resolve outlier data because
of how it dominates the Euclidean norm.

The first case are the low-rank approximations without outliers in the input data

\begin{figure}[H]
  \centering
  \includegraphics[scale=0.5]{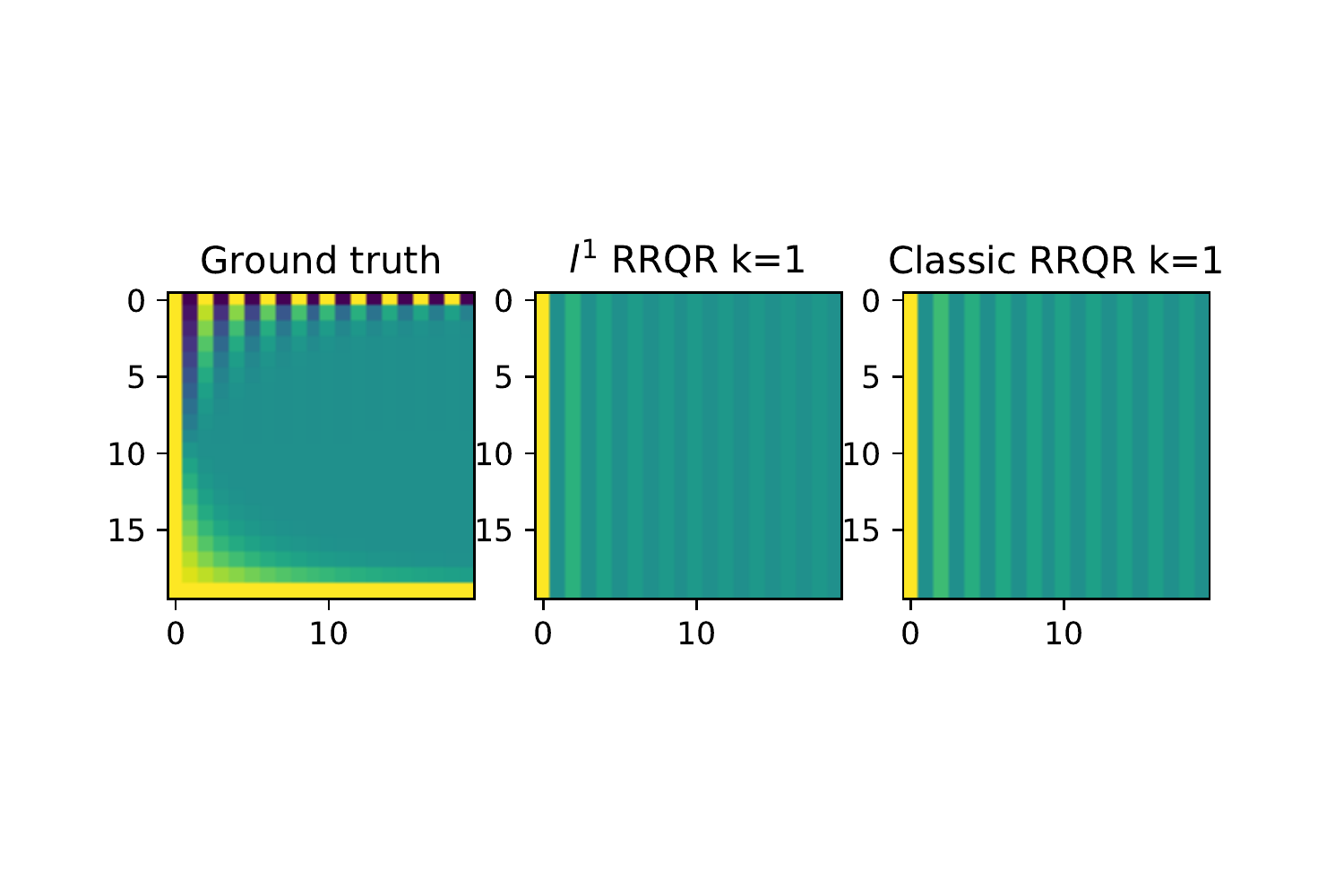}
  \includegraphics[scale=0.5]{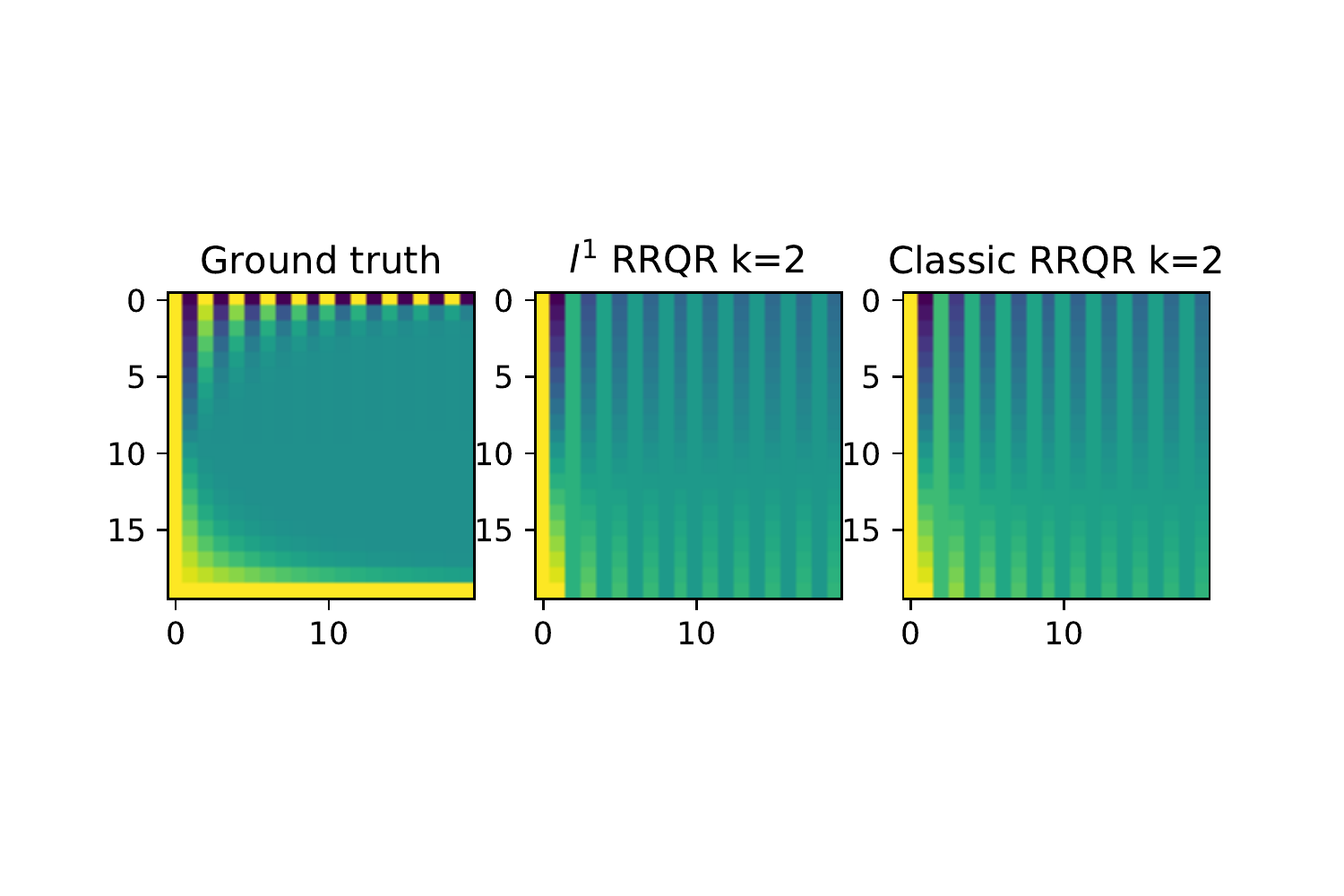}
  \includegraphics[scale=0.5]{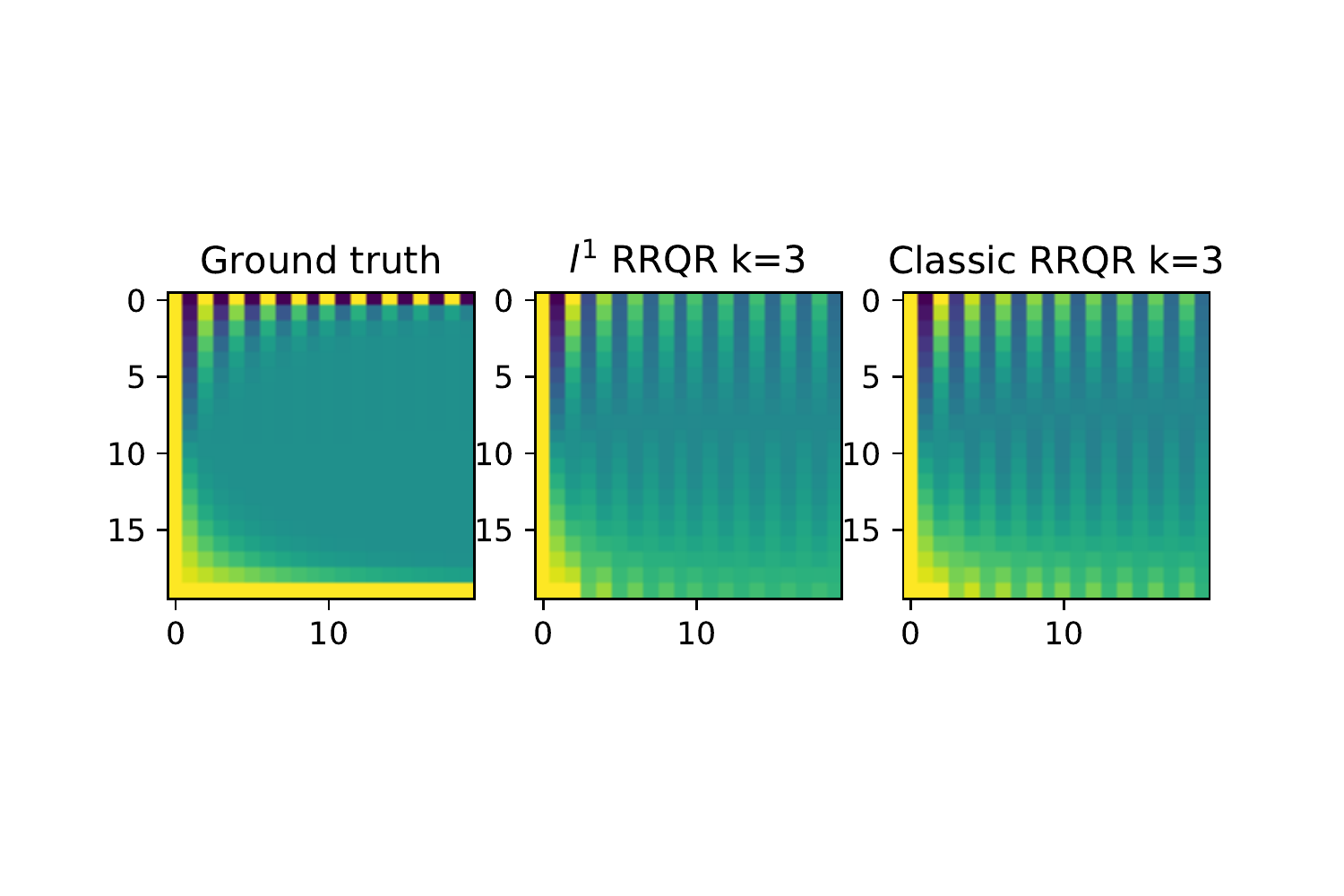}
  \caption{Low rank approximations of clean data using both classical and $l^1$ RRQR factorizations}
  \label{fig:sidebyside}
\end{figure}

The next case I introduce two outliers of much larger magnitude than surrounding data. The
classical RRQR quickly gravitates to the columns containing these outliers 
because the outlier data gets squared in the Euclidean norm and then dominates it. 
In the $l^1$ norm however this effect is much less pronounced. See figure \ref{fig:sidebyside2}
below.

\begin{figure}[H]
  \centering
  \includegraphics[scale=0.5]{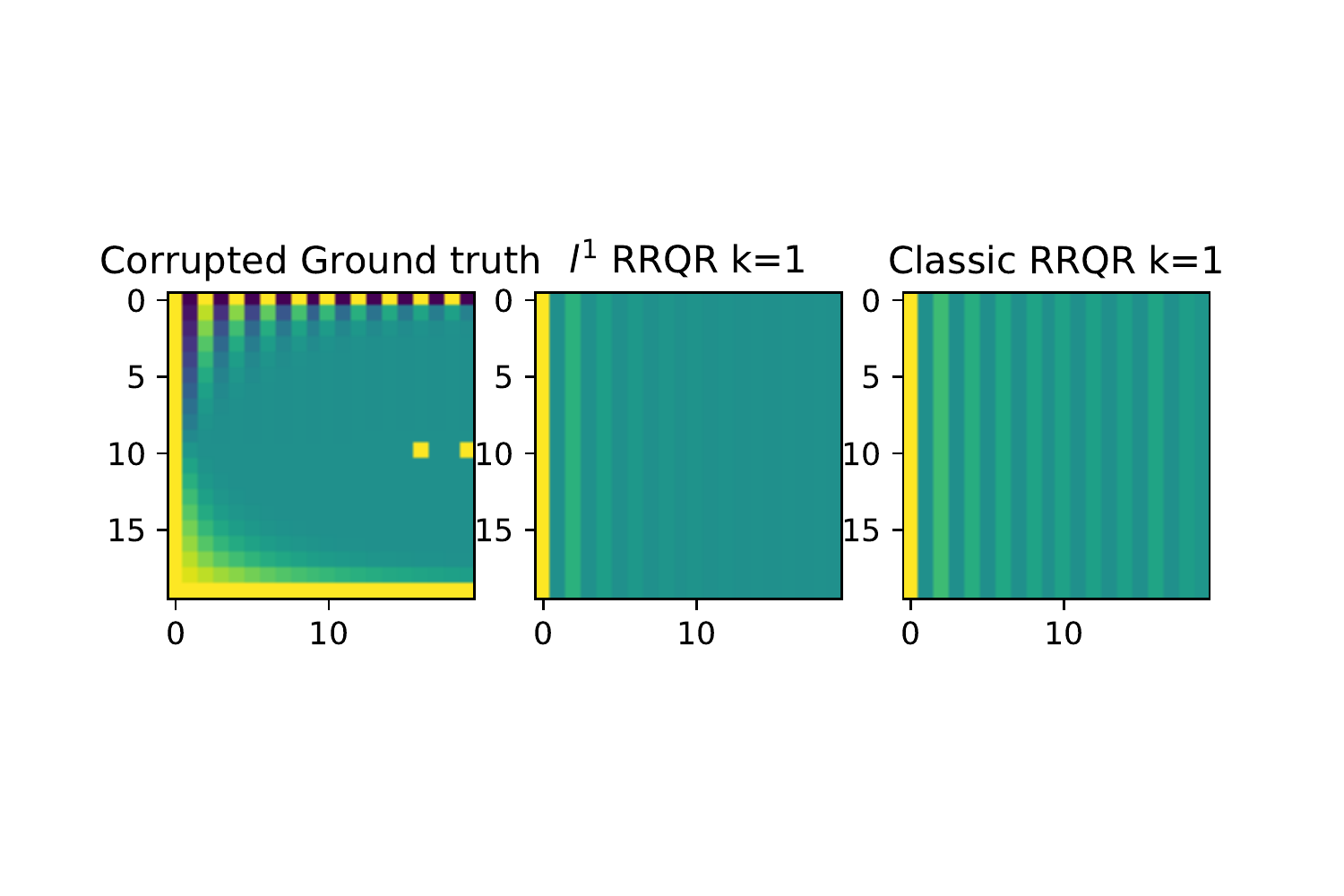}
  \includegraphics[scale=0.5]{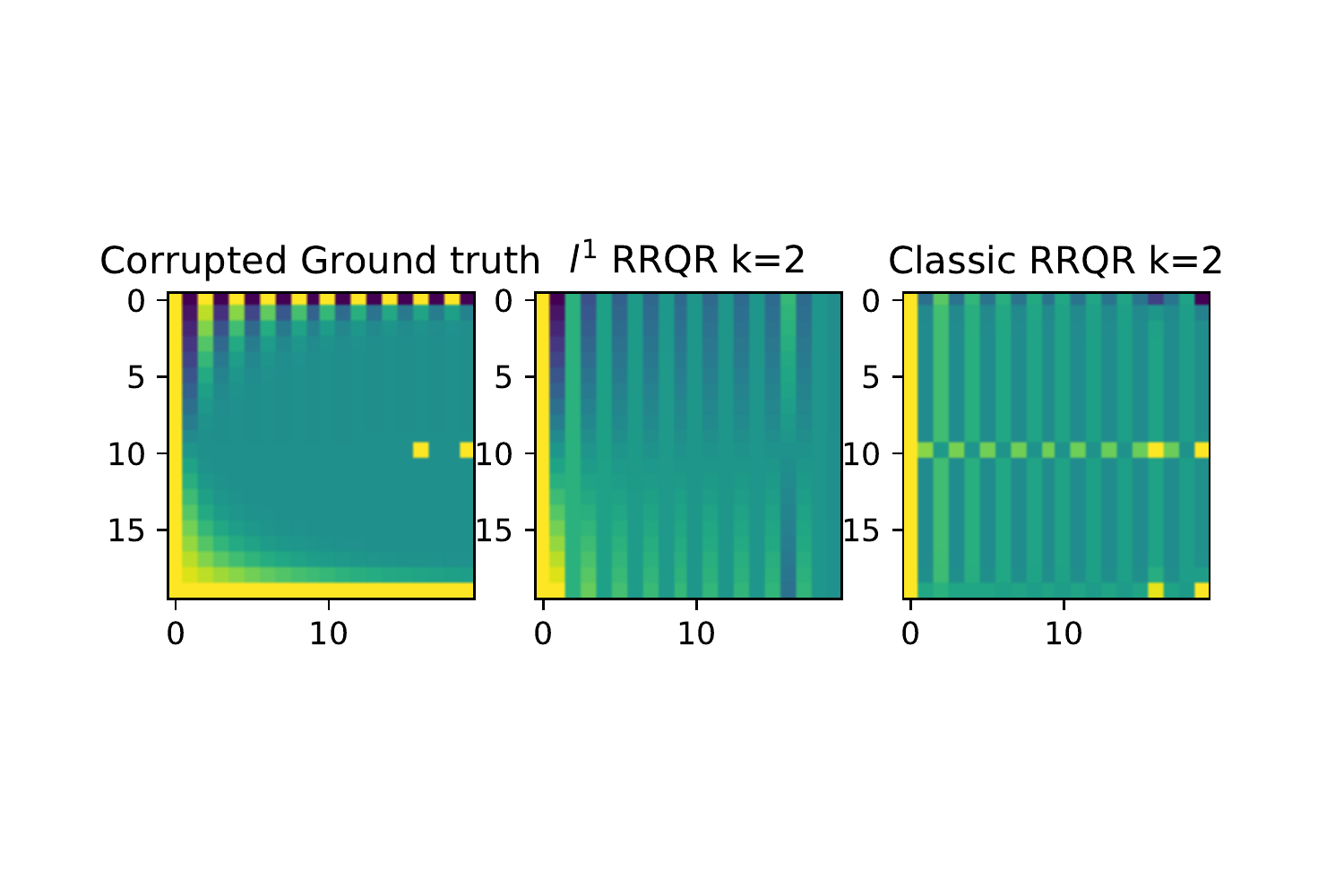}
  \includegraphics[scale=0.5]{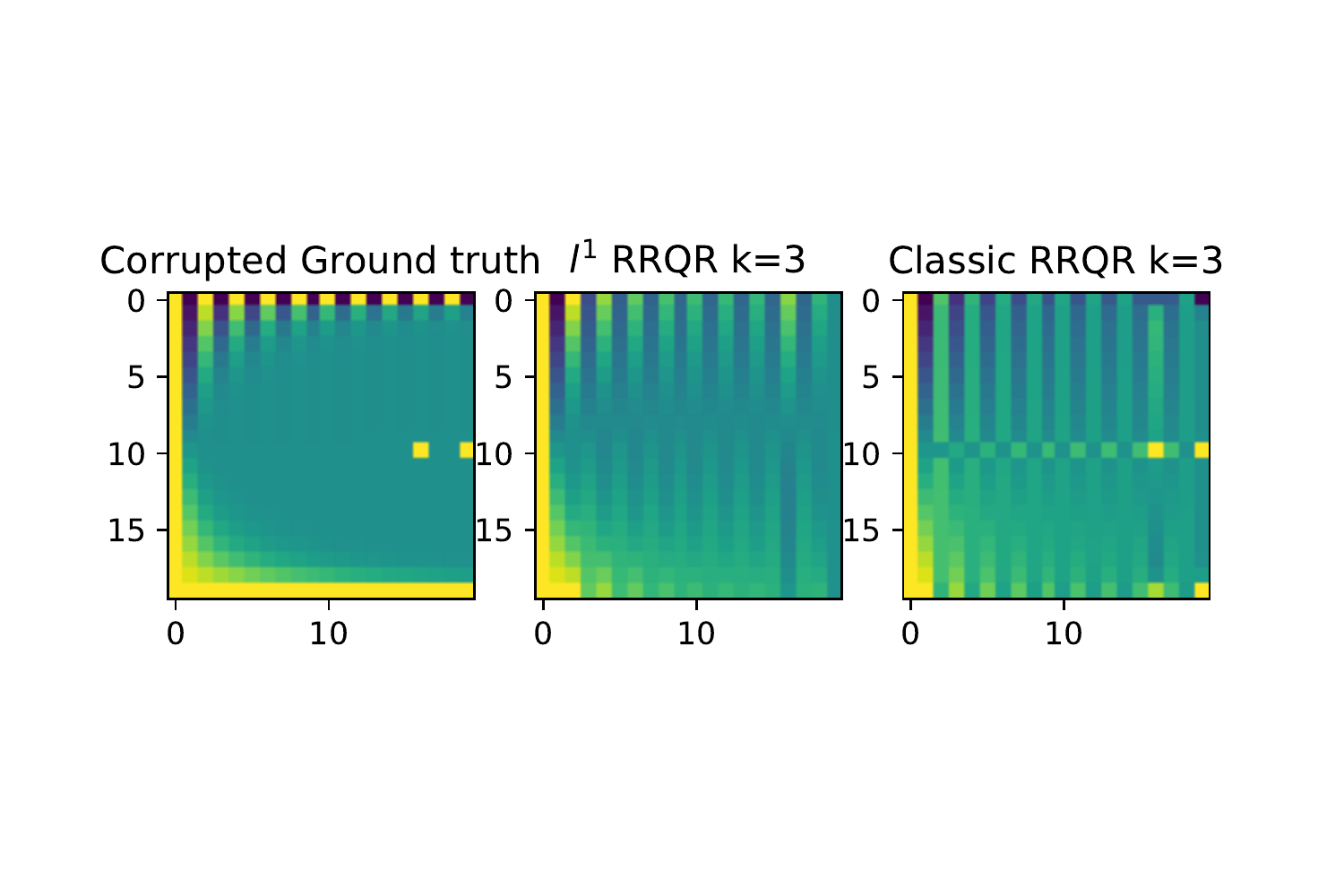}
  \caption{Low rank approximations of corrupted data using both classical and $l^1$ RRQR factorizations}
  \label{fig:sidebyside2}
\end{figure}

\section{Conclusions}
\label{sec:conclusions}

I derived a rank-revealing factorization that shares some similarities
to classic rank-revealing QR with column pivoting. Instead of 
the $Q$ factor being orthogonal it has conditioning that is independent
of the conditioning of $A$. Furthermore the rank-revealing factorization
presented here does not depend strictly on using dot products and the
$l^2$ norm. I validated that claim by implementing the algorithm
for the $l^1$ norm case, where least-norm-solution is equivalent to a linear program. 

While I was able to numerically validate the conditioning properties of $Q$ 
I was unable to mathematically prove them. The key fact to be proven
remains conjecture (\ref{con:goodcondition}). Without orthogonality
properties available in the $l^2$ case most avenues for
proof are lost. I believe however that careful use of the optimality
properties for the least-norm solution $c_j$ (see eqns \ref{eq:recurrence})
may be able to overcome the loss of orthogonality.

\appendix
\section{Python Implementations}

This section contains python implementations for the key algorithms of this paper. 
I give three implementations here. The first two, implementations
\ref{impl:opensource} and \ref{impl:nag} solve the same problem, and may be used
interchangeably in the following implementation \ref{impl:rrqr} which actually
computes the rank-revealing factorization. 

The first algorithm solves linear systems in the least-$l^1$-norm sense by translating
it to a linear program and then using SciPy. This has the advantage of only requiring
open source tools that are readily available on the internet.

\begin{figure}[H]
\begin{lstlisting}[language=Python] 
import numpy as np
import scipy.optimize as opt
def lst1norm(A,b):
    (m,n)=A.shape
    nvars=m+n
    ncons=2*m
    cons=np.zeros((ncons,nvars))
    cons[0:m,0:m]=-np.identity(m)
    cons[0:m,m:m+n]=A
    cons[m:2*m,0:m]=-np.identity(m)
    cons[m:2*m,m:m+n]=-A
    c=np.zeros(nvars)
    c[0:m]=1.0
    ub=np.zeros(ncons)
    ub[0:m]=b
    ub[m:2*m]=-b
    bounds=[]
    for i in range(0,m):
        bounds.append( (0,None) )
    for i in range(m,m+n):
        bounds.append((None,None))
    out=opt.linprog(c,cons,ub,
    None,None,bounds,
    options={'maxiter' : 10000,'tol':1e-6},
    method='interior-point')
    return out
\end{lstlisting}
  \caption{Solve $Ax=b$ in the minimum-$l^1$ sense using open source tools}
  \label{impl:opensource}
\end{figure}

The next Python implementation makes use of the NAG library 
routine \emph{e02ga} through the Nag Library for Python \cite{naglibrarypython}.
Full documentation for this routine may be found at \cite{naglonefit}. This
relies on the closed-source NAG library, but since the \emph{e02ga} routine
is specialized to the least-$l^1$-norm problem it is significantly faster
than a generic linear-programming approach as shown above. This is important
for the rank revealing factorization because it spends almost all of its
time solving linear systems in this minimum-norm sense. This enables
factorizing much larger matrices.

\begin{figure}[H]
\begin{lstlisting}[language=Python]
import numpy as np
from naginterfaces.library.fit import glin_l1sol
def lst1norm_nag_glin(A,b):
    (m,n)=A.shape
    B=np.zeros((m+2,n+2))
    B[0:m,0:n]=A
    (_,_,soln,objf,rank,i)=glin_l1sol(B,b)
    return (soln[0:n],objf)
\end{lstlisting}
  \caption{Solve $Ax=b$ in the minimum-$l^1$ sense using the NAG library}
  \label{impl:nag}
\end{figure}

Finally the actual factorization. As mentioned above, this factorization depends
on the ability to solve linear systems in the least-norm sense. Since I gave
two possible ways to achieve this, I made the least-norm-solver an input argument
which may be changed either to the fully open source solver, or to the faster
NAG-based solver.

\begin{figure}[H]
\begin{lstlisting}[language=Python]
import numpy as np
def l1rrqr(A,tol=1e-15,l1alg=lst1norm_nag_glin):
    (m,n)=A.shape
    Q=np.zeros((m,n))
    R=np.zeros((n,m))
    norms=[np.linalg.norm(A[:,i],ord=1) for i in range(0,n)]
    k=np.argmax(norms)
    perm=[k]
    sout=set(perm)
    sin=set([i for i in range(0,n)]).difference(sout)
    i=0
    Q[:,i]=A[:,k]/np.linalg.norm(A[:,k],ord=1)
    R[i,i]=np.linalg.norm(A[:,k],ord=1)
    while sin:
        i=i+1
        V=Q[:,0:i]
        vals=[l1alg(V,A[:,i])[1] for i in sin]
        ids=[i for i in sin]
        k=ids[np.argmax(vals)]
        sout.add(k)
        sin=sin.difference(sout)
        perm.append(k)
        (c,objf)=l1alg(V,A[:,k])
        y=A[:,k]-V@c
        Q[:,i]=y/np.linalg.norm(y,ord=1)
        R[0:i,i]=c
        R[i,i]=np.linalg.norm(y,ord=1)
    return (Q,R,perm)
\end{lstlisting}
  \caption{A Python implementation of the arbitrary-norm rank-revealing factorization}
  \label{impl:rrqr}
\end{figure}

\bibliographystyle{siamplain}
\bibliography{mylib}
\end{document}